\documentclass{amsart}

\usepackage{amsmath}

\newcommand{\donothing}[1]{}
\newcommand{\Comment}{\donothing}

\newcommand{\Commentsoff}{\renewcommand{\Comment}{\donothing}}
\Commentsoff

\makeatletter
\def\subsection{\@startsection{subsection}{2}%
  \z@{.5\linespacing\@plus.7\linespacing}{.5\linespacing}%
  {\normalfont\bfseries}}
\makeatother

\usepackage{amssymb}

\usepackage{enumerate}

\usepackage[all]{xy}
\SilentMatrices %

\newtheorem{definition}{Definition}[section]
\newtheorem{lemma}[definition]{Lemma}
\newtheorem{proposition}[definition]{Proposition}
\newtheorem{corollary}[definition]{Corollary}
\newtheorem{theorem}[definition]{Theorem}

\newtheorem{remark}[definition]{Remark}
\newtheorem{example}[definition]{Example}

\DeclareMathOperator*{\colim}{colim}

\newcommand{\ra}{\to}

\DeclareMathOperator{\obj}{Obj}

\DeclareMathOperator{\im}{Im}
\DeclareMathOperator{\id}{id}

\DeclareMathOperator{\card}{card}

\newcommand{\Diff}{{\mathfrak{D}\mathrm{iff}}}
\newcommand{\Top}{{\mathfrak{T}\mathrm{op}}}
\newcommand{\Set}{{\mathfrak{S}\mathrm{et}}}

\newcommand{\DS}{{\mathcal{DS}}}
\newcommand{\cD}{{\mathcal{D}}}
\newcommand{\cM}{{\mathcal{M}}}

\newcommand{\oo}{\infty}
\newcommand{\sse}{\subseteq}
\newcommand{\Ci}{C^{\infty}}

\def \N{\mathbb{N}}
\def \Z{\mathbb{Z}}
\def \Q{\mathbb{Q}}

\def \R{\mathbb{R}}

\newcommand{\blank}{-}

\newcommand{\dfn}[1]{\textbf{#1}}

\usepackage{ifpdf}
\ifpdf  \usepackage[pdftex,bookmarks=false]{hyperref}
\else   \usepackage[hypertex]{hyperref}
\fi

\begin{document}

\title{The $D$-topology for diffeological spaces}
\author{J. Daniel Christensen}
\address{Department of Mathematics,
University of Western Ontario,
London, ON N6A 5B7, Canada}
\email{jdc@uwo.ca}
\author{Gord Sinnamon}
\address{Department of Mathematics,
University of Western Ontario,
London, ON N6A 5B7, Canada}
\email{sinnamon@uwo.ca}
\author{Enxin Wu}
\address{Faculty of Mathematics, 
University of Vienna, 
Oskar-Morgenstern-Platz 1,
1090 Vienna, Austria}
\email{enxin.wu@univie.ac.at}
\date{September 14, 2015}

\begin{abstract}
Diffeological spaces are generalizations of smooth manifolds which include
singular spaces and function spaces.
For each diffeological space, Iglesias-Zemmour introduced a natural topology called the $D$-topology.
However, the $D$-topology has not yet been studied seriously in the existing literature.
In this paper, we develop the basic theory of the $D$-topology for diffeological spaces.
We explain that the topological spaces that arise as the $D$-topology of a diffeological space
are exactly the $\Delta$-generated spaces and give results and examples
which help to determine when a space is $\Delta$-generated.
Our most substantial results show how the $D$-topology on the function
space $C^{\infty}(M,N)$ between smooth manifolds compares to other well-known topologies.
\end{abstract}


\maketitle

\tableofcontents

\section{Introduction}

Smooth manifolds are some of the most important objects in mathematics.
They contain a wealth of geometric information, such as
tangent spaces, tangent bundles, differential forms, de~Rham cohomology, etc., and
this information can be put to great use in proving theorems and making calculations.
However, the category of smooth manifolds and smooth maps
is not closed under many useful constructions,
such as subspaces, quotients, function spaces, etc.
On the other hand, various convenient categories of topological spaces
are closed under these constructions,
but the geometric information is missing.
Can we have the best of both worlds?

Since the 1970's, the category of smooth manifolds has been enlarged in several
different ways to a well-behaved category as described above, and these
approaches are nicely summarized and compared in~\cite{St}.
In this paper, we work with diffeological spaces, 
which were introduced by J.~Souriau in the early 1980s~\cite{So1,So2},
and in particular we study the natural topology that any diffeological space has.

A diffeological space is a set $X$ along with a specified set of maps $U \to X$
for each open set $U$ in $\R^n$ and each $n \in \N$, satisfying a presheaf condition,
a sheaf condition, and a non-triviality condition (see Definition~\ref{de:diffeological-space}).
Given a diffeological space $X$, the $D$-topology on $X$ is the largest topology making
all of the specified maps $U \to X$ continuous.
In this paper, we make the first detailed study of the $D$-topology.
Our results include theorems giving properties and characterizations of
the $D$-topology as well as many examples which show the behaviour that
can occur and which rule out some natural conjectures.

Our interest in these topics comes from several directions.
First, it is known~\cite{SYH} that the topological spaces which arise as the $D$-topology of
diffeological spaces are precisely the $\Delta$-generated spaces, which were introduced by
Jeff Smith as a possible convenient category for homotopy theory and were studied by~\cite{D,FR}.
Some of our results help to further understand which spaces are $\Delta$-generated,
and we include illustrative examples.

Second, for any diffeological spaces $X$ and $Y$, the set $C^{\infty}(X,Y)$ of smooth maps
from $X$ to $Y$ is itself a diffeological space in a natural way and thus can be endowed
with the $D$-topology.
Since the topology arises completely canonically, it is instructive to compare
it with other topologies that arise in geometry and analysis when $X$ and $Y$
are taken to be smooth manifolds.
A large part of this paper is devoted to this comparison, and again
we give both theorems and illustrative examples.

Finally, this paper arose from work on the homotopy theory of diffeological
spaces~\cite{CW}, and can be viewed as the topological groundwork for this project.
It is for this reason that we need to focus on an approach that produces
a well-behaved category, rather than working with a theory of infinite-dimensional
manifolds, such as the one thoroughly developed in the book~\cite{KM}.
We will, however, make use of results from~\cite{KM}, as many of the underlying
ideas are related.

\medskip

Here is an outline of the paper, with a summary of the main results:

In Section~\ref{se:background}, we review some basics of diffeological spaces.
For example, we recall that
the category of diffeological spaces is complete, cocomplete and cartesian closed,
and that it contains the category of smooth manifolds as a full subcategory.
Moreover, like smooth manifolds, every diffeological space is formed by gluing together open subsets of $\R^n$,
with the difference that $n$ can vary and that the gluings are not necessarily via diffeomorphisms.

In Section~\ref{se:D-topology}, we study the $D$-topology of a diffeological space,
which was introduced by Iglesias-Zemmour in~\cite{I1}.
We show that the $D$-topology is determined by the smooth curves (Theorem~\ref{th:curves}),
while diffeologies are not (Example~\ref{diffeology}).
We recall a result of~\cite{SYH} which says
that the topological spaces arising as the $D$-topology of a diffeological space
are exactly the $\Delta$-generated spaces (Proposition~\ref{delta}).
We give a necessary condition and a sufficient condition
for a space to be $\Delta$-generated (Propositions~\ref{La} and~\ref{countable})
and show that neither is necessary and sufficient (Proposition~\ref{pr:fc} and Example~\ref{lpc}).
We can associate two topologies to a subset of a diffeological space.
We discuss some conditions under which the two topologies coincide
(Lemmas~\ref{l1} and~\ref{l2}, Proposition~\ref{pr:topologies-agree}, and Corollary~\ref{D-open}).

Section~\ref{se:function-spaces} contains our most substantial results.
We compare the $D$-topology on function spaces between smooth manifolds with other well-known topologies.
The results are (1) the $D$-topology is almost always strictly finer
than the compact-open topology (Proposition~\ref{pr:D-vs-compact-open} and Example~\ref{compare});
(2) the $D$-topology is always finer than the weak topology (Proposition~\ref{weakvsD}) 
and always coarser than the strong topology (Theorem~\ref{th:D-vs-strong});
(3) we give a full characterization of the
$D$-topology as the smallest $\Delta$-generated topology containing
the weak topology (Theorem~\ref{conj:D-top});
(4) as a consequence, we show that the weak topology is equal to the $D$-topology
if and only if the weak topology is locally path-connected (Corollary~\ref{iff:weak=D});
(5) in particular, when the codomain is $\R^n$ or the domain is compact,
the $D$-topology coincides with the weak topology
(Corollary~\ref{cor:N=R} and Corollary~\ref{M-cpt}),
but not always (Example~\ref{weak-neq-D}).
\medskip

All smooth manifolds in this paper are assumed to be Hausdorff, finite-dimensional, second countable and without boundary.
\medskip

We would like to thank Andrew Stacey and
Chris Schommer-Pries for very helpful conversations, 
and Jeremy Brazas for the idea behind Example~\ref{lpc}.

\section{Background on diffeological spaces}\label{se:background}

Here is some background on diffeological spaces.
While we often cite early sources, almost all of the material
in this section is in the book~\cite{I2}, which we recommend as a good reference.

\begin{definition}[\cite{So2}]\label{de:diffeological-space}
A \dfn{diffeological space} is a set $X$
together with a specified set $\mathcal{D}_X$ of functions $U \ra X$ (called \dfn{plots})
for each open set $U$ in $\R^n$ and for each $n \in \N$,
such that for all open subsets $U \subseteq \R^n$ and $V \subseteq \R^m$:
\begin{enumerate}
\item (Covering) Every constant map $U \ra X$ is a plot;
\item (Smooth Compatibility) If $U \ra X$ is a plot and $V \ra U$ is smooth,
then the composition $V \ra U \ra X$ is also a plot;
\item (Sheaf Condition) If $U=\cup_i U_i$ is an open cover
and $U \ra X$ is a set map such that each restriction $U_i \ra X$ is a plot,
then $U \ra X$ is a plot.
\end{enumerate}
We usually use the underlying set $X$ to denote the diffeological space $(X,\mathcal{D}_X)$.
\end{definition}

\begin{definition}[\cite{So2}]
Let $X$ and $Y$ be two diffeological spaces,
and let $f:X \rightarrow Y$ be a set map.
We say that $f$ is \dfn{smooth} if for every plot $p:U \ra X$ of $X$,
the composition $f \circ p$ is a plot of $Y$.
\end{definition}

The collection of all diffeological spaces and smooth maps forms a category,
which we denote $\Diff$.
Given two diffeological spaces $X$ and $Y$,
we write $C^\infty(X,Y)$ for the set of all smooth maps from $X$ to $Y$.
An isomorphism in $\Diff$ will be called a \dfn{diffeomorphism}.

Every smooth manifold $M$ is canonically a diffeological space
with the same underlying set and plots taken to be all smooth maps $U \ra M$ in the usual sense.
We call this the \dfn{standard diffeology} on $M$.
By using charts, it is easy to see that smooth maps in the usual sense between smooth manifolds
coincide with smooth maps between them with the standard diffeology.
This gives the following standard result, which can be found, for
example, in Section~4.3 of~\cite{I2}.

\begin{theorem}
There is a fully faithful functor from the category of smooth manifolds to $\Diff$.
\end{theorem}

From now on, unless we say otherwise,
every smooth manifold considered as a diffeological space is equipped with the standard diffeology.

\begin{proposition}[\cite{I1}]
Given a set $X$,
let $\mathcal{D}$ be the set of all diffeologies on $X$ ordered by inclusion.
Then $\mathcal{D}$ is a complete lattice.
\end{proposition}

This follows from the fact that $\mathcal{D}$ is closed under arbitrary intersections.
The largest element in $\mathcal{D}$ is called the \dfn{indiscrete diffeology} on $X$,
for which every function $U \ra X$ is a plot,
and the smallest element in $\mathcal{D}$ is called the \dfn{discrete diffeology} on $X$,
for which the plots $U \to X$ are the locally constant maps.

The smallest diffeology on $X$ containing a set of 
maps $A=\{U_i \ra X\}_{i \in I}$ is called the diffeology \dfn{generated} by $A$.
It consists of all maps $f:V \ra X$
such that there exists an open cover $\{V_j\}$ of $V$
with the property that $f$ restricted to each $V_j$ is either constant or factors through some element 
$U_i \ra X$ in $A$ via a smooth map $V_j \ra U_i$.
The standard diffeology on a smooth manifold is generated by any smooth atlas on the manifold.
For every diffeological space $X$, $\mathcal{D}_X$ is generated by 
$\cup_{n \in \N} C^\infty(\R^n,X)$.

Generalizing the previous paragraph,
let $A=\{f_j:X_j \ra X\}_{j \in J}$ be a set of functions from some 
diffeological spaces to a fixed set $X$.
Then there exists a smallest diffeology on $X$ making all $f_j$ smooth,
and we call it the \dfn{final diffeology} defined by $A$.
For a diffeological space $X$ with an equivalence relation $\sim$,
the final diffeology defined by the quotient map $\{X \twoheadrightarrow X/{\sim}\}$
is called the \dfn{quotient diffeology}.
Similarly, let $B=\{g_k:Y \ra Y_k\}_{k \in K}$ be a set of functions 
from a fixed set $Y$ to some diffeological spaces.
Then there exists a largest diffeology on $Y$ making all $g_k$ smooth,
and we call it the \dfn{initial diffeology} defined by $B$.
For a diffeological space $X$ and a subset $A$ of $X$,
the initial diffeology defined by the inclusion map $\{A \hookrightarrow X\}$ 
is called the \dfn{subset diffeology}.
More generally, we have the following well-known result:

\begin{theorem}
The category $\Diff$ is both complete and cocomplete.
\end{theorem}

This is proved in~\cite{BH}, but can be found implicitly in earlier work.
We give a brief sketch here.
The forgetful functor $\Diff \ra \Set$ to the category of sets preserves both
limits and colimits since it has both left and right adjoints, 
given by the discrete and indiscrete diffeologies.
The diffeology on the (co)limit is the initial (final) diffeology defined by
the natural maps.
In more detail, let $F: J \ra \Diff$ be a functor from a small category $J$
and write $\bar{F}$ for the composite $J \ra \Diff \ra \Set$.
Then $U \ra \lim \bar{F}$ is a plot if and only if the composite
$U \ra \lim \bar{F} \ra \bar{F}(j)$ is a plot of $F(j)$ for each $j \in \obj (J)$.
It is not hard to check directly that $\lim \bar{F}$ with this diffeology is $\lim F$.
Similarly, $p: U \ra \colim \bar{F}$ is a plot if and only if there is an open cover
$\{U_i\}$ of $U$ such that the restriction $p|_{U_i}$ factors as
$U_i \ra \bar{F}(j) \ra \colim \bar{F}$ for some $j \in \obj (J)$, with the first map
a plot of $F(j)$.
It is not hard to check directly that $\colim \bar{F}$ with this diffeology is $\colim F$.

The category of diffeological spaces also enjoys another convenient property:

\begin{theorem}[\cite{I1}]\label{th:cartesian-closed}
The category $\Diff$ is cartesian closed.
\end{theorem}

Given two diffeological spaces $X$ and $Y$,
the set of maps $\{U \ra C^\infty(X,Y) \mid U \times X \ra Y$ is smooth$\}$ 
forms a diffeology on $C^\infty(X,Y)$.
We call it the \dfn{functional diffeology} on $C^\infty(X,Y)$,
and we always equip hom-sets with the functional diffeology.
Furthermore, for each diffeological space $Y$,
$\blank \times Y:\Diff \rightleftharpoons \Diff :C^\infty(Y,\blank)$ is an adjoint pair.
\medskip

A smooth manifold of dimension $n$ is formed by gluing together
some open subsets of $\R^n$ via diffeomorphisms.
A diffeological space is also formed by gluing together open subsets
of $\R^n$ (with the standard diffeology) via smooth maps,
possibly for all $n \in \N$.
To make this precise, we introduce the following concept:

Let $\DS$ be the category with
objects all open subsets of $\R^n$ for all $n \in \N$
and morphisms the smooth maps between them.
Given a diffeological space $X$,
we define $\DS/X$ to be the category with objects all plots of $X$ and 
morphisms the commutative triangles
\[
\xymatrix@C=5pt{U \ar[dr]_p \ar[rr]^f & & V \ar[dl]^q \\ & X, }
\]
with $p,q$ plots of $X$ and $f$ a smooth map.
We call $\DS/X$ the \dfn{category of plots of $X$}.
It is equipped with a forgetful functor $F:\DS/X \ra \Diff$
sending a plot $U \ra X$ to $U$ regarded as a diffeological space
and sending the morphism displayed above to $f$.
We can use $F$ to show that any diffeological space $X$ can be built out of Euclidean spaces:

\begin{proposition}\label{colim}
The colimit of the functor $F:\DS/X \ra \Diff$ is $X$.
\end{proposition}

\begin{proof}
Clearly there is a natural cocone $F \ra X$ sending the above commutative triangle to itself.
For each diffeological space $Y$ and cocone $g:F \ra Y$,
we define a set map $h:X \ra Y$ by sending $x \in X$ to $g(x)(\R^0)$,
where by abuse of notation the second $x$ denotes the plot $\R^0 \ra X$ with image $x \in X$.
Note that $h$ induces a (unique) cocone map
since $h(p(u))=g(p(u))=g(p) \circ u$ for each plot $p:U \ra X$ and each $u \in U$,
 which also implies the smoothness of $h$.
\end{proof}

The result is essentially the same as~\cite[Exercise~33]{I2}.
\medskip

Given a diffeological space $X$, the category $\DS/X$ can be used to define
geometric structures on $X$.
See~\cite{I2,So3,La} for a discussion of differential forms and the de Rham cohomology of
a diffeological space,
and see~\cite{He,La} for tangent spaces and tangent bundles.

\section{The $D$-topology}\label{se:D-topology}

We can associate to every diffeological space the following interesting topology:

\begin{definition}[{\cite{I1}, \cite[Chapter~2]{I2}}]
Given a diffeological space $X$, the final topology induced by its plots,
where each domain is equipped with the standard topology,
is called the \dfn{$D$-topology} on $X$.
\end{definition}

In more detail, if $(X, \mathcal{D})$ is a diffeological space,
then a subset $A$ of $X$ is open in the $D$-topology of $X$
if and only if $p^{-1}(A)$ is open for each $p \in \mathcal{D}$.
We call such subsets \dfn{$D$-open}.
If $\mathcal{D}$ is generated by a subset $\mathcal{D}'$,
then $A$ is $D$-open if and only if $p^{-1}(A)$ is open for each $p \in \mathcal{D}'$.

A smooth map $X \ra X'$ is continuous when $X$ and $X'$ are equipped with the
$D$-topology, and so this defines a functor $D: \Diff \ra \Top$ to the
category of topological spaces.

\begin{example}
(1) The $D$-topology on a smooth manifold with the standard diffeology
coincides with the usual topology on the manifold.

(2) The $D$-topology on a discrete diffeological space is discrete,
and the $D$-topology on an indiscrete diffeological space is indiscrete.
\end{example}

Every topological space $Y$ has a natural diffeology,
called the \dfn{continuous diffeology}, whose plots $U \ra Y$
are the continuous maps.
This was defined in Section~2.8 of~\cite{Do}.
A continuous map $Y \ra Y'$ between topological spaces is smooth when $Y$ and $Y'$ are equipped
with the continuous diffeology, and so this defines a functor
$C: \Top \ra \Diff$.

\begin{proposition}\label{DC}
The functors $D: \Diff \rightleftharpoons \Top :C$ are adjoint,
and we have $C \circ D \circ C = C$ and $D \circ C \circ D = D$.
\end{proposition}

\begin{proof}
The adjointness is~\cite[Proposition~3.1]{SYH}, and the rest is easy.
\end{proof}

\begin{proposition}[\cite{He,La}]\label{La}
For each diffeological space, the $D$-topology is locally path-connected.
\end{proposition}

However, not every locally path-connected space comes from a diffeological space;
see Example~\ref{lpc}.

\subsection{The $D$-topology is determined by smooth curves}\label{sse:curves}

\begin{definition}
We say that a sequence $x_m$ in $\R^n$ \dfn{converges fast} to $x$ in $\R^n$
if for each $k \in \N$ the sequence $m^k(x_m-x)$ is bounded.
\end{definition}

Note that every convergent sequence has a subsequence which converges fast.

\begin{lemma}[Special Curve Lemma {\cite[page~18]{KM}}]\label{lem:scl}
Let $x_m$ be a sequence which converges fast to $x$ in $\R^n$.
Then there is a smooth curve $c:\R \ra \R^n$
such that $c(t)=x$ for $t \leq 0$, $c(t)=x_1$ for $t \geq 1$,
$c(\frac{1}{m})=x_m$ for each $m \in \Z^+$,
and $c$ maps $[\frac{1}{m+1},\frac{1}{m}]$ to the line segment joining $x_{m+1}$ and $x_m$.
\end{lemma}

\begin{theorem}\label{th:curves}
The $D$-topology on a diffeological space $X$ is determined by $C^\infty(\R,X)$,
in the sense that a subset $A$ of $X$ is $D$-open
if and only if $p^{-1}(A)$ is open for every $p \in C^\infty(\R,X)$.
\end{theorem}

\begin{proof}
($\Rightarrow$) This follows from the definition of the $D$-topology.

($\Leftarrow$) Suppose that $p^{-1}(A)$ is open for every $p \in C^\infty(\R,X)$.
Consider a plot $q : U \ra X$, and let $x \in q^{-1}(A)$.
Suppose that $\{x_m\}$ converges fast to $x$.
By the Special Curve Lemma, there is a smooth curve $c: \R \ra U$ such that
$c(\frac{1}{m})=x_m$ for each $m$ and $c(0)=x$.
Since $c^{-1}(q^{-1}(A))$ is open, $x_m$ is in $q^{-1}(A)$ for $m$ sufficiently large.
So $q^{-1}(A)$ is open in $U$.
\end{proof}

\begin{example}\label{diffeology}
Let $X$ be $\R^2$ with the standard diffeology,
and let $Y$ be the set $\R^2$ with the diffeology generated by $C^\infty(\R,\R^2)$.
Then $D(X)$ is homeomorphic to $D(Y)$ since $C^\infty(\R,X) = C^\infty(\R,Y)$,
but $X$ and $Y$ are not diffeomorphic since the identity map $\R^2 \ra \R^2$
does not locally factor through curves.
In other words, the $D$-topology is determined by smooth curves, but the diffeology is not.

In this example, $Y$ has the smallest diffeology such that $C^{\infty}(\R, \R^2)$
consists of the usual smooth curves.
In contrast, by Boman's theorem~\cite[Corollary~3.14]{KM}, $X$ has the \emph{largest}
diffeology such that $C^{\infty}(\R, \R^2)$ consists of the usual smooth curves.
That is, $p:U \ra X$ is a plot if and only if for every smooth function $c:\R \ra U$,
the composite $p \circ c$ is in $C^\infty(\R,X)$.
\end{example}

\subsection{Relationship with $\Delta$-generated topological spaces}\label{sse:delta-generated}

Write $\Delta^n$ for the standard $n$-simplex in $\Top$.

\begin{definition}
A topological space $X$ is called \dfn{$\Delta$-generated} if the following condition holds:
$A \subseteq X$ is open if and only if $f^{-1}(A)$ is open in $\Delta^n$
for each continuous map $f:\Delta^n \ra X$ and each $n \in \N$.
\end{definition}

It is not hard to show that being $\Delta$-generated is the same as being $\R$-generated
or $[0, 1]$-generated,
i.e., that one can determine the open sets of a $\Delta$-generated space using just the
continuous maps $\R \ra X$ or $[0, 1] \ra X$.
This follows from the existence of a surjective continuous map $\R \ra \Delta^n$
that exhibits $\Delta^n$ as a quotient of $\R$.
Note the similarity to Theorem~\ref{th:curves}.
More on $\Delta$-generated topological spaces can be found in~\cite{D,FR}.

\begin{proposition}[\cite{SYH}]\label{delta}
The spaces in the image of the functor $D$ are exactly the $\Delta$-generated topological spaces.
\end{proposition}

Since the argument is easy, we include a proof.

\begin{proof}
Let $X$ be a diffeological space, and consider $A \subseteq D(X)$.
Suppose $f^{-1}(A)$ is open in $\R$ for all continuous $f: \R \ra D(X)$.
Then $f^{-1}(A)$ is open in $\R$ for all smooth $f: \R \ra X$.
Thus $A$ is open in $D(X)$, and so $D(X)$ is $\Delta$-generated.

Now suppose that $Y$ is $\Delta$-generated.
By adjointness, the identity map $D(C(Y)) \ra Y$ is continuous.
We claim that it is a homeomorphism, and so $Y$ is in the image of $D$.
Indeed, suppose $A \subseteq D(C(Y))$ is open.
That is, $f^{-1}(A)$ is open in $\R$ for all smooth $f: \R \ra C(Y)$.
That is, $f^{-1}(A)$ is open in $\R$ for all continuous $f: \R \ra Y$.
Then, since $Y$ is $\Delta$-generated, $A$ is open in $Y$.
\end{proof}

Because of this, it will be helpful to better understand which topological
spaces are $\Delta$-generated.

\begin{proposition}\label{countable}
Every locally path-connected first countable topological space is $\Delta$-gen\-er\-ated.
\end{proposition}

\begin{proof}
Let $(X,\tau)$ be a locally path-connected first countable topological space.
Then for each $x \in X$, there exists a neighborhood basis $\{A_i\}_{i=1}^\infty$ of $x$,
such that
\begin{enumerate}
\item each $A_i$ is path-connected; and
\item $A_{i+1} \subseteq A_i$.
\end{enumerate}
This is because, for a neighborhood basis $\{B_i\}_{i=1}^\infty$ of $x$,
we can define $A_1$ to be the path-component of $B_1$ containing $x$,
and $A_i$ to be the path-component of $A_{i-1} \cap B_i$ containing $x$ for $i \geq 2$.
Since $X$ is locally path-connected, each $A_i$ is open.

Now let $\tau'$ be the final topology on $X$
for all continuous maps $\Delta^n \ra (X,\tau)$ for all $n \in \N$.
Clearly $\tau \subseteq \tau'$.
Suppose $A$ is not in $\tau$.
This means that there exists $x \in A$ such that
for each $U \in \tau$ which is a neighborhood of $x$,
there exists $x_U \in U \setminus A$.
Let $\{A_i\}_{i=1}^\infty$ be a neighbourhood basis for $x$ with the above two properties,
and write $x_n \in A_n \setminus A$ accordingly.
Define $f:[0,1] \ra X$ by letting $f|_{[\frac{1}{i+1},\frac{1}{i}]}$ be a continuous path
connecting $x_{i+1}$ to $x_i$ in $A_i$, and $f(0)=x$.
It is easy to see that $f$ is continuous for $(X,\tau)$,
but $f^{-1}(A)$ is not open in $[0,1]$.
So $A$ is not in $\tau'$.
\end{proof}

It follows from Propositions~\ref{La} and~\ref{delta} that
every $\Delta$-generated space is locally path-connected.
However, not every $\Delta$-generated space is first countable:

\begin{proposition}\label{pr:fc}
Let $X$ be a set with the complement-finite topology.
We write $\card(X)$ for its cardinality.
Then
\begin{enumerate}
\item $X$ is $\Delta$-generated if $\card(X)<\card(\N)$ or $\card(X) \geq \card(\R)$;
\item $X$ is not $\Delta$-generated if $\card(X)=\card(\N)$.
\end{enumerate}
\end{proposition}

Note that $X$ is not first countable when $\card(X) \geq \card(\R)$.
This provides a counterexample to the converse of Proposition~\ref{countable}.

\begin{proof}
(1) If $X$ is a finite set,
then the complement finite topology is the discrete topology.
Hence $X$ is $\Delta$-generated.

Assume $\card(X) \geq \card(\R)$,
and let $B$ be a non-closed subset of $X$,
that is, $B \neq X$ and $\card(B) \geq \card(\N)$.
We must construct a continuous map $f:\R \ra X$ such that $f^{-1}(B)$ is not closed in $\R$.
Note that in this case, every injection $\R \ra X$ is continuous.

Take an injection $\tilde{f}:\{\frac{1}{n}\}_{n \in \Z^+} \ra B$.
We can extend this to an injection $f:\R \ra X$ with $f(0) \in X \setminus B$.
This map is what we are looking for.

(2) If $\card(X)=\card(\N)$, then every continuous map $[0, 1] \ra X$ is constant.
Otherwise, since every point in $X$ is closed,
$[0, 1]$ would be a disjoint union of at least $2$ and at most countably many non-empty closed subsets,
which contradicts a theorem of Sierpi\'{n}ski (see, e.g.,~\cite[A.10.6]{vM}
or the slick argument posted by Gowers~\cite{G}).
Since $X$ is not discrete, it is not $\Delta$-generated.
\end{proof}

\begin{remark}
Assume the continuum hypothesis.
Then the above proposition says that a set $X$ with the complement finite topology
is $\Delta$-generated if and only if $X$ is not an infinite countable set.
\end{remark}

Here is an example showing that
not every locally path-connected topological space is the $D$-topology of a diffeological space:

\begin{example}\label{lpc}
As a set, let $X$ be the disjoint union of copies of the closed unit interval
indexed by the set $J$ of countable ordinals.
We write elements in $X$ as $x_a$ with $x \in [0,1]$ and $a \in J$.
Let $Y$ be the quotient set $X/{\sim}$,
where the only non-trivial relations are $1_a \sim 1_b$ for all $a, b \in J$.
Since we will only work with $Y$,
we denote the elements of $Y$ in the same way as those of $X$.
The topology on $Y$ is generated by the following basis:
\begin{enumerate}
\item the open interval $(x_a,y_a)$ for each $0 \leq x < y \leq 1$ and $a \in J$;
\item the set $U_{a,x} := (\cup_{a \leq b \in J}[0_b,1_b]) \cup (\cup_{c<a} (x_c,1_c])$
for each $a \in J$ and $x \in [0,1)$.
\end{enumerate}
One can show that $Y$ is locally path-connected (but not first countable).
However, $Y$ is not $\Delta$-generated.
Indeed, let $A=\cup_{a \in J}(0_a,1_a]$.
Then $A$ is not open in $Y$.
For every continuous map $f:\Delta^n \ra X$, we claim that $f^{-1}(A)$ is open in $\Delta^n$.
Otherwise, there exists $u \in f^{-1}(A)$ such that no open neighborhood of $u$ is contained in $f^{-1}(A)$.
Since the intervals $(x_a, y_a)$ are open, we must have $f(u)=1_a$, the common point.
Choose a sequence $(u_i)$ converging to $u$
such that each $u_i$ is not in $f^{-1}(A)$.
Then $f(u_i) = 0_{b_i}$ for some countable ordinals $b_i$.
Let $b$ be a countable ordinal larger than each $b_i$.
Then $U_{b,0}$ is an open set containing $f(u)$ but none of the $f(u_i)$, so
$f(u_i)$ is not convergent to $f(u)=1_a$,
which contradicts the continuity of $f$.
\end{example}

\subsection{Two topologies related to a subset of a diffeological space}\label{sse:subset}

Let $X$ be a diffeological space,
and let $Y$ be a quotient set of $X$.
Then we can give $Y$ two topologies:
\begin{enumerate}
\item the $D$-topology of the quotient diffeology on $Y$;
\item the quotient topology of the $D$-topology on $X$.
\end{enumerate}
Since $D:\Diff \ra \Top$ is a left adjoint,
these two topologies are the same.

\medskip

Similarly, let $X$ be a diffeological space,
and let $A$ be a subset of $X$.
Then we can give $A$ two topologies:
\begin{enumerate}
\item $\tau_1(A):$ the $D$-topology of the subset diffeology on $A$;
\item $\tau_2(A):$ the sub-topology of the $D$-topology on $X$.
\end{enumerate}
However, these two topologies are not always the same.
In general, we can only conclude that $\tau_2(A) \subseteq \tau_1(A)$.

\begin{example}\label{2Q}
(1) Let $A$ be a subset of $\R$.
Then $\tau_1(A)$ is discrete if and only if $A$ is totally disconnected under the sub-topology of $\R$.
In particular, if $A=\Q$,
then $\tau_1(\Q)$ is the discrete topology,
which is strictly finer than the sub-topology $\tau_2(\Q)$.

(2) Let $f:\R \ra \R$ be a continuous and nowhere differentiable function,
and let $A=\{(x,f(x)) \mid x \in \R\}$ be its graph, equipped with the subset diffeology of $\R^2$.
Then $\tau_1(A)$ is the discrete topology,
which is strictly finer than the sub-topology of $\R^2$.
Here is the proof.
Let $g:\R \ra \R^2$ be a smooth map whose image is in $A$,
and define $y, z: \R \to \R$ by $g(t) = (y(t),z(t))$.
Assume that $y'(a) \neq 0$ for some $a \in \R$.
Then by the inverse function theorem, $y:\R \ra \R$ is a local diffeomorphism around $a$.
Since $\im(g) \subseteq A$, we have $z=f \circ y$,
which implies that $f=z \circ y^{-1}$ around $y(a)$, contradicting nowhere differentiability of $f$.
Therefore, any plot of the form $\R \ra A$ is constant.
By Theorem~\ref{th:curves}, $\tau_1(A)$ is discrete.
On the other hand, the sub-topology $\tau_2(A)$ is homeomorphic to the usual topology on $\R$.
\end{example}

\begin{definition}[{\cite[2.14]{I2}}]\label{de:embedded}
When $\tau_1(A) = \tau_2(A)$, we say that $A$ is an \dfn{embedded subset} of $X$.
\end{definition}

We are interested in conditions under which this holds.

\begin{lemma}\label{l1}
Let $A$ be a convex subset of $\R^n$.
Then $A$ is an embedded subset of $\R^n$.
\end{lemma}

\begin{proof}
Following the idea of the proof of~\cite[Lemma~24.6(3)]{KM},
let $B \subseteq A$ be closed in the $\tau_1(A)$-topology,
and let $\bar{B}$ be the closure of $B$ in $A$ for the $\tau_2(A)$-topology.
Note that the $\tau_2(A)$-topology is the same as the sub-topology of $\R^n$.
Hence, for any $b \in \bar{B}$, we can find a sequence $b_n$ in $B$ which converges fast to $b$.
Since $A$ is convex, the Special Curve Lemma (Lemma~\ref{lem:scl}) says that
there is a smooth curve $c:\R \ra A$ such that $c(0)=b$ and $c(\frac{1}{n})=b_n$ for each $n \in \Z^+$.
Therefore, $b \in B$ by the definition of the $D$-topology.
\end{proof}

\begin{lemma}\label{l2}
If $A$ is a $D$-open subset of a diffeological space $X$, then 
$A$ is an embedded subset of $X$.
\end{lemma}

\begin{proof}
Let $B$ be in $\tau_1(A)$.
To show that $B$ is in $\tau_2(A)$, it suffices to show that $B$ is $D$-open in $X$.
Let $p:U \ra X$ be an arbitrary plot of $X$.
Since $A$ is $D$-open in $X$, $p^{-1}(A)$ is an open subset of $U$.
Hence the composition of $p^{-1}(A) \hookrightarrow U \ra X$ is also a plot for $X$,
which factors through the inclusion map $A \hookrightarrow X$.
Since $B \in \tau_1(A)$, $(p |_{p^{-1}(A)})^{-1}(B)$ is open in $p^{-1}(A)$,
which implies that $p^{-1}(B)$ is open in $U$.
Thus $B$ is $D$-open in $X$, as required.
\end{proof}

\begin{example}
$GL(n,\R)$ is $D$-open in $M(n,\R) \cong \R^{n^2}$, so it is an embedded subset.
\end{example}

Also see Corollary~\ref{D-open} for another example.
Note that
Lemma~\ref{l2} is not true if we change $D$-open to $D$-closed:

\begin{example}
Let $A=\{\frac{1}{n}\}_{n \in \Z^+} \cup \{0\} \subset \R$.
Then $A$ is $D$-closed in $\R$.
It is easy to check that $\tau_1(A)$ is discrete,
and is strictly finer than $\tau_2(A)$.
\end{example}

\begin{proposition}\label{pr:topologies-agree}
Let $X$ be a diffeological space and let $A$ be a subset of $X$.
If there exists a $D$-open neighborhood $C$ of $A$ in $X$
together with a smooth retraction $r:C \ra A$,
then $A$ is embedded in $X$.
(Here both $C$ and $A$ are equipped with the subset diffeologies from $X$.)
\end{proposition}

\begin{proof}
Let $B \in \tau_1(A)$.
Then $r^{-1}(B) \in \tau_1(C)=\tau_2(C)$ is $D$-open in $X$.
Therefore, $B=A \cap r^{-1}(B) \in \tau_2(A)$.
\end{proof}

\begin{example}
Given a smooth manifold $M$ of dimension $n>0$,
by the strong Whitney Embedding Theorem,
there is a smooth embedding $M \hookrightarrow \R^{2n}$.
If we view $M$ as a subset of $\R^{2n}$, then it is an embedded subset,
since there is an open tubular neighborhood $U$ of $M$ in $\R^{2n}$
together with a smooth retraction $U \ra M$.
\end{example}

\section{The $D$-topology on function spaces}\label{se:function-spaces}

Let $M$ and $N$ be smooth manifolds.
Recall that the set $C^\infty(M,N)$ of smooth maps from $M$ to $N$
has a functional diffeology described just after Theorem~\ref{th:cartesian-closed}.
In this section, we consider the topological space obtained by taking the
$D$-topology associated to this diffeology,
and we compare it to other well-known topologies on this set:
the compact-open topology, the weak topology, and the strong topology.

Here is a review of these three topologies and their relationship.
The books~\cite{Hi,KM,Mi} are good references for the weak and strong topologies.

The compact-open topology on $C^\infty(M,N)$ has a subbasis which consists of the sets
$A(K,W) = \{f \in C^\infty(M,N) \mid f(K) \subseteq W\}$,
where $K$ is a non-empty compact subset of $M$ and $W$ is an open subset of $N$.
(This makes sense for any diffeological spaces $M$ and $N$, where $K$ is then
required to be compact in $D(M)$ and $W$ to be open in $D(N)$.)

We now describe a subbasis for the weak topology on $C^\infty(M,N)$.
For $r \in \N$, $(U,\phi)$ a chart of $M$, $(V,\psi)$ a chart of $N$,
$K \subseteq U$ compact, $f \in C^\infty(M,N)$ with $f(K) \subseteq V$, and $\epsilon > 0$,
we define the set
$N^r(f,(U,\phi),(V,\psi),K,\epsilon)$ to be
$\{g \in C^\infty(M,N) \mid g(K) \subseteq V \text{ and }
\|D^i(\psi \circ f \circ \phi^{-1})(x)-D^i(\psi \circ g \circ \phi^{-1})(x)\|<\epsilon$
for each $x \in \phi(K)$ and each multi-index $i$ with $|i| \leq r\}$.
These sets form a subbasis for the weak topology.
Here $i=(i_1,\ldots,i_m)$ is a multi-index in $\N^m$ with $m=\dim(M)$, $|i|=i_1+\cdots+i_m$,
and $D^i$ is the differential operator
$\frac{\partial^{|i|}}{\partial x_1^{i_1} \cdots \partial x_m^{i_m}}$.

A subbasis for the strong topology on $C^\infty(M,N)$ is similar, but
it allows constraints using multiple charts.
More precisely, if
$N^r(f,(U_i,\phi_i),(V_i,\psi_i),K_i,\epsilon_i)$ is a family of subbasic
sets for the weak topology such that the collection $\{ U_i \}$ is locally finite,
then the intersection of this family is a subbasic set for the strong topology.
In fact, one can show that these intersections form a base for the strong topology.

Each of these is at least as fine as the previous one,
i.e.,
\[
\text{compact-open topology} \subseteq \text{weak topology} \subseteq \text{strong topology.}
\]
The first inclusion is proved in Lemma~\ref{le:wcco}, and the second is clear.
The compact-open topology and the weak topology coincide if and only if $M$ or $N$
is zero-dimensional (see Example~\ref{compare}).
Moreover, the weak topology and the strong topology coincide if the domain
$M$ is compact and are different if $M$ is non-compact and $N$ has positive dimension
(see~\cite[pp.~35--36]{Hi}).

\medskip

Now we start our comparison of the $D$-topology with these topologies. 
The following lemma is needed for the subsequent proposition.

\begin{lemma}\label{le:product}
Let $X$ and $Y$ be two diffeological spaces such that $D(X)$ is locally compact Hausdorff.
Then the natural bijection $D(X \times Y) \ra D(X) \times D(Y)$ is a homeomorphism.
\end{lemma}

Note that when $X$ is a smooth manifold, $D(X)$ is locally compact Hausdorff.

\begin{proof}
First observe that the natural bijection $D(U \times V) \ra D(U) \times D(V)$ is a homeomorphism
for $U$ and $V$ open subsets of Euclidean spaces, since in this case the $D$-topology
is the usual topology.
The functors $D: \Diff \to \Top$,
$Z \times -: \Diff \to \Diff$ for any diffeological space $Z$
and $W \times -: \Top \to \Top$ for any locally compact Hausdorff space $W$
all preserve colimits since they are left adjoints.
Thus the claim follows from Proposition~\ref{colim}, using that
$D(X)$ is locally compact Hausdorff, as is each $D(U)$ for $U$ an open subset of some Euclidean space.
\end{proof}

For general $X$ and $Y$, one can show using a similar argument that
the $D$-topology on $D(X \times Y)$ corresponds under the bijection above
to the smallest $\Delta$-generated topology containing the product topology
on $D(X) \times D(Y)$.

\begin{proposition}\label{pr:D-vs-compact-open}
For diffeological spaces $X$ and $Y$,
the $D$-topology on $C^\infty(X,Y)$ contains the compact-open topology.
\end{proposition}

This result is a stepping stone to proving the stronger statement
that the $D$-topology contains the weak topology.

\begin{proof}
Recall that the compact-open topology has a subbasis which consists of the sets
$A(K,W) = \{f \in C^\infty(X,Y) \mid f(K) \subseteq W\}$,
where $K$ is a non-empty compact subset of $D(X)$ and $W$ is an open subset of $D(Y)$.
We will show that each $A(K,W)$ is $D$-open.
Let $\phi: U \ra C^\infty(X,Y)$ be a plot of $C^\infty(X,Y)$.
Since the corresponding map $\bar{\phi}: U \times X \ra Y$ is smooth,
$\bar{\phi}^{-1}(W)$ is open in $D(U \times X)$.
So for each $u \in \phi^{-1}(A(K,W))$,
$\{u\} \times K$ is in the open set $\bar{\phi}^{-1}(W)$.
Note that the natural map $D(U \times X) \ra D(U) \times D(X)$ is a 
homeomorphism by Lemma~\ref{le:product}.
By the compactness of $K$ and the definition of the product topology,
$V \times K \subseteq \bar{\phi}^{-1}(W)$ for some open neighborhood $V$ of $u$ in $U$,
which implies that $\phi^{-1}(A(K,W))$ is open in $U$.
Thus $A(K,W)$ is open in the $D$-topology.
\end{proof}

We will see in Example~\ref{compare} that the $D$-topology is almost
always strictly finer than the compact-open topology.

The next lemma will be used to show that the $D$-topology contains the
weak topology for function spaces between smooth manifolds.

\begin{lemma}\label{differential}
Let $U$ be an open subset in $\R^n$ and let $i$ be a multi-index in $\N^n$.
Then $D^i:C^\infty(U,\R) \ra C^\infty(U,\R)$ is smooth.
\end{lemma}

\begin{proof}
Let $\phi:V \ra C^\infty(U,\R)$ be a plot with $\dim(V)=m$.
This means that the associated map $\bar{\phi}:V \times U \ra \R$ 
defined by $\bar{\phi}(v,u)=\phi(v)(u)$ is smooth.
Write $j$ for the multi-index $(0_m,i) \in \N^{m+n}$, with $0_m$ a sequence of $m$ zeros.
Then $D^j(\bar{\phi}):V \times U \ra \R$ is smooth.
Since $D^j(\bar{\phi})(v,u)=D^i(\phi(v))(u)$,
$D^i \circ \phi$ is a plot,
which implies the smoothness of $D^i$.
\end{proof}

Note that the smoothness of $D^i$ does not imply its continuity in general.
It is an easy exercise that for $|i| > 0$ and $n > 0$, $D^i$ is not continuous in the 
compact-open topology but is continuous in both the weak and strong topologies.
\medskip

Now we can compare the $D$-topology with the weak topology for 
function spaces between smooth manifolds:

\begin{proposition}\label{weakvsD}
Let $M$ and $N$ be smooth manifolds.
Then the $D$-topology on $C^\infty(M,N)$ contains the weak topology.
\end{proposition}

\begin{proof}
Recall that the weak topology on $C^\infty(M,N)$ has the sets
$N^r(f,(U,\phi),(V,\psi),K,\epsilon)$,
described at the beginning of Section~\ref{se:function-spaces}, as a subbasis.

Let $p:W \ra C^\infty(M,N)$ be a plot, that is, $\bar{p}:W \times M \ra N$ 
given by $\bar{p}(w,x)=p(w)(x)$ is smooth.
If $w \in p^{-1}(N^r(f,(U,\phi),(V,\psi),K,\epsilon))$,
then by Proposition~\ref{pr:D-vs-compact-open}, Lemma~\ref{differential} and
the facts that $\phi$ and $\psi$ are diffeomorphisms, 
only finitely many differentials are considered, $K$ is compact and $V$ is open,
it is not hard to see that there exists an open neighborhood $W'$ of $w$ in $W$
such that $W' \subseteq p^{-1}(N^r(f,(U,\phi),(V,\psi),K,\epsilon))$.
Therefore, $N^r(f,(U,\phi),(V,\psi),K,\epsilon)$ is $D$-open.
\end{proof}

Since the weak topology is almost always strictly finer than the
compact-open topology, so is the $D$-topology:

\begin{example}\label{compare}
The $D$-topology on $C^\infty(\R,\R)$ is strictly finer than the compact-open topology.
To prove this, consider $U=N^1(\hat{0},(\R,\id),(\R,\id),[-1,1],1)$,
where $\hat{0}$ is the zero function.
This is open in the weak topology and thus is open in the $D$-topology.
We claim that no open neighborhood of $\hat{0}$ in the compact-open topology
of $C^\infty(\R,\R)$ is contained in $U$.
Otherwise, we may assume $\hat{0} \in A(K,(-\epsilon,\epsilon)) \subseteq U$
for some $\epsilon > 0$ and some compact $K$,
since if $\hat{0} \in A(K_1,W_1) \cap \cdots \cap A(K_m,W_m)$,
then $0 \in W_i$ for each $i$ and
$\hat{0} \in A(K_1 \cup \cdots \cup K_m, W_1 \cap \cdots \cap W_m) 
\subseteq A(K_1,W_1) \cap \cdots \cap A(K_m,W_m)$.
Then clearly $f: \R \ra \R$ defined by $f(x) = (\epsilon/2) \sin(2x/\epsilon)$
is in $A(K,(-\epsilon,\epsilon))$ for any $K$.
But $f$ is not in $U$ since $f'(0) = 1$.

Using a similar argument, with bump functions, one can show that when
$M$ and $N$ are smooth manifolds of dimension at least 1, then the
weak topology is strictly finer than the compact-open topology.
Thus the $D$-topology is strictly finer than the compact-open topology
in this situation.
\end{example}

In general, the weak topology is different from the $D$-topology on $C^\infty(M,N)$:

\begin{example}\label{weak-neq-D}
(1) Let $\N$ and $\{0,1\}$ be equipped with the discrete diffeologies.
Let $f:\N \ra \{0,1\}$ be the constant function sending everything to $0$,
and let $f_n:\N \ra \{0,1\}$ be defined by $f_n^{-1}(0)=\{0,1,\ldots,n\}$.
Note that $f_n$ converges to $f$ in the weak topology for the following reason.
Since each element in the subbasis of the weak topology
depends only on the values of the function and its derivatives on a compact subset of $\N$,
any of them containing $f$ must contain all $f_n$ for $n$ large enough.

On the other hand, we claim that for each $n$ there is no continuous path
$F : [0, 1] \to C^{\infty}(\N, \{0, 1\})$ with $F(0) = f_n$ and $F(1) = f$,
where the codomain is given the weak topology.
Since the weak topology contains the compact-open topology, such an $F$
gives rise to a continuous function $[0, 1] \times \N \to \{0, 1\}$,
i.e., a homotopy from $D(f_n)$ to $D(f)$.
Since these maps are clearly not homotopic, no such $F$ exists.

Thus the weak topology is not locally path-connected.
It follows from Proposition~\ref{La} that the weak topology is
different from the $D$-topology on $C^\infty(\N,\{0,1\})$.

The above argument in fact shows that every continuous path in
$C^{\infty}(\N, \{0, 1\})$ with respect to a topology containing
the compact-open topology is constant.
In particular, this holds for the $D$-topology, and since the
$D$-topology is $\Delta$-generated, it must be discrete.

(2) Let $X$ be a countable disjoint union of copies of $S^1$,
i.e., $X=\coprod_{i \in \N} X_i$ with each $X_i=S^1$.
Then the weak topology on $C^\infty(X,S^1)$ is not locally path-connected,
by a similar argument,
with $f:X \ra S^1$ defined by $f|_{X_i}=\id:X_i \ra S^1$,
and $f_n:X \ra S^1$ defined by
\[
f_n|_{X_i}=\begin{cases} \id, & \textrm{if $i=0,1,\ldots,n$} \\
                        -\id, & \textrm{otherwise.}
           \end{cases}
\]

(3) The weak topology on $C^\infty(\R^2 \setminus (\{0\} \times \Z),\, S^1)$ 
is not locally path-connected, by a similar argument,
with $f:\R^2 \setminus (\{0\} \times \Z) \ra S^1$ defined by
\[
f(x,y)=\frac{1-e^{2\pi (x+iy)}}{|1-e^{2\pi (x+iy)}|},
\]
and $f_n:\R^2 \setminus (\{0\} \times \Z) \ra S^1$ defined by
\[
f_n(x,y)=f(x,\phi_n(y)),
\]
where $\phi_n:\R \ra \R$ is a strictly increasing smooth function with
$\phi_n(t) = t$ for $|t| \leq n$ and $|\phi_n(t)| < n+1$ for all $t$.
\end{example}

The above examples all show that the weak topology is not locally path
connected and in particular that it is not $\Delta$-generated.
The $D$-topology is a $\Delta$-generated topology which contains
the weak topology, and the following theorem says that, given this, it is 
as close to the weak topology as possible.

\begin{theorem}\label{conj:D-top}
Let $M$ and $N$ be smooth manifolds.
Then the $D$-topology on $C^{\infty}(M,N)$ is the smallest $\Delta$-generated
topology containing the weak topology.
\end{theorem}

\begin{proof}
First note that by Proposition~\ref{weakvsD}, the $D$-topology contains the weak topology,
and by Proposition~\ref{delta}, the $D$-topology is $\Delta$-generated.
So we must prove that the $D$-topology on $C^\infty(M,N)$ is contained in 
every $\Delta$-generated topology containing the weak topology.

So let $\tau$ be a $\Delta$-generated topology containing the weak topology
and assume that $A \subseteq C^\infty(M,N)$ not open in $\tau$. 
Since $\tau$ is $\Delta$-generated, there is a $\tau$-continuous map
$p : \R \to \Ci(M,N)$ such that $p^{-1}(A)$ is not open in $\R$.
Since $\tau$ contains the weak topology, $p$ is weak continuous.
By composing with a translation in $\R$, we can assume that $0$
is a non-interior point of $p^{-1}(A)$.
Thus we can find a sequence $t_r$ of real numbers converging to $0$
so that $p(t_r) \not\in A$ for each $r$.
By Theorem~\ref{th:weak-scl}, there is a smooth curve 
$q : \R \to \Ci(M,N)$ such that $q(2^{-j}) = p(t_{r_j}) \not\in A$ 
for each $j$ and $q(0) = p(0)$.
This shows that $A$ is not open in the $D$-topology.
\end{proof}

Since every $\Delta$-generated space is locally path connected
(see Propositions~\ref{La} and~\ref{delta}),
the previous result is in fact a special case of the next result.

\begin{theorem}\label{th:lpc}
Let $M$ and $N$ be smooth manifolds.
Then the $D$-topology on $C^{\infty}(M,N)$ is the smallest locally path connected
topology containing the weak topology.
\end{theorem}

\begin{proof}
Suppose $\tau$ is a locally path-connected topology that
contains the weak topology, $A$ is not $\tau$-open, and
$f \in A$ is not $\tau$-interior to $A$.
Since the weak topology on $\Ci(M,N)$ is first countable,
there is a countable weak neighborhood basis $(W_r)_{r=1}^\oo$ of $f$.
Contained in each $W_r$ there is a path-connected $\tau$-neighborhood $T_r$ of $f$.
For each $r$, choose an $f_r \in T_r \setminus A$ and a $\tau$-continuous
(and therefore weak continuous) path from $f$ to $f_r$ lying entirely in $T_r \sse W_r$.
We can concatenate these paths to produce a weak continuous path $p$
such that $p(0) = f$ and $p(2^{-r}) = f_r$.
By Theorem~\ref{th:weak-scl}, there is a smooth curve $q : \R \to \Ci(M,N)$
such that $q(0) = f$ and $q(2^{-j}) = f_{r_j}$.
Then $q^{-1}(A)$ contains $0$ but not $2^{-j}$ for any $j$, so $A$ is not
open in the $D$-topology.
\end{proof}

As a corollary, we have the following necessary and sufficient
condition for the weak topology to be equal to the $D$-topology:

\begin{corollary}\label{iff:weak=D}
Let $M$ and $N$ be smooth manifolds. 
Then the weak topology on $C^\infty(M,N)$ coincides with the $D$-topology 
if and only if the weak topology is locally path-connected.
\end{corollary}

\begin{proof}
This follows from Theorem~\ref{th:lpc} (or from Theorem~\ref{conj:D-top},
using that the weak topology is second countable~\cite[pp.~35--36]{Hi}).
\end{proof}

This allows us to give a situation in which the $D$-topology and the weak topology coincide.
(See also Corollary~\ref{M-cpt}.)

\begin{corollary}\label{cor:N=R}
Let $M$ be a smooth manifold.
Then the weak topology on $C^\infty(M,\R^n)$ coincides with the $D$-topology.
\end{corollary}

\begin{proof}
By Lemma~\ref{le:convex}, the weak topology on $\Ci(M,\R^n)$ has a basis of
convex sets.  A linear path is smooth and hence weak continuous, so it 
follows that this topology is locally path connected.
\end{proof}

Our next goal is to show that the $D$-topology is contained in the strong
topology.  
We first need some preliminary results.

\begin{lemma}\label{le:locally-convex}
Let $M$ be a smooth manifold and let $N$ be an open subset of $\R^d$. 
Then the $D$-topology on $\Ci(M,N)$ is contained in any topology 
that contains the weak topology and has a basis of convex sets.
\end{lemma}

Here we say that a subset of $\Ci(M,N)$ is \dfn{convex} if it is
convex when regarded as a subset of the real vector space $\Ci(M,\R^d)$.

\begin{proof}
A convex set isn't necessarily path connected, since linear paths may
not be continuous.
Thus Theorem~\ref{th:lpc} doesn't apply directly.
However, in the proof of Theorem~\ref{th:lpc}, all that is used is
that the subsets $T_r$ are path connected in the \emph{weak} topology.
Since linear paths are smooth, they are weak continuous, and so the
proof goes through.
\end{proof}

\begin{lemma}\label{le:strong-open-subspace}
\hspace*{-1pt}Let $M$ be a smooth manifold and let $N$ be an open subset of $\R^d$. 
Then $\Ci(M,N)$ is an open subspace of $\Ci(M,\R^d)$ when both are
equipped with the strong topology.
\end{lemma}

\begin{proof}
We first prove that the strong topology on $C^\oo(M,N)$ is the subspace
topology of the strong topology on $C^\oo(M,\R^d)$.
Since the inclusion map $N \to \R^d$ induces a continuous map in the
strong topologies (\cite[Exercise 10(b), page~65]{Hi}), the intersection
of a strong open set in $C^\oo(M,\R^d)$ with $C^\oo(M,N)$ is open in $C^\oo(M,N)$.
On the other hand, the data for each weak subbasic set $A$ in $C^\oo(M,N)$ defines
a weak subbasic set in $C^\oo(M,\R^d)$ whose intersection with $C^\oo(M,N)$ is $A$.
Since the strong subbasic sets are certain intersections of the weak subbasic
sets, our claim follows.

Now we show that $\Ci(M,N)$ is an open subset of $\Ci(M,\R^d)$, following
the argument in Lemma~\ref{le:wcco}.
For $f \in \Ci(M,N)$, choose charts for $M$ and $N$ and compact sets
$K_i \sse M$ as described in Lemma~\ref{le:charts}(b).
Then 
\[ f \in \cap_{i=1}^\oo \, N^0(f, (U_i, \phi_i), (N, \id), K_i, 1) \sse \Ci(M,N) , \] 
where each $N^0(f, (U_i, \phi_i), (N, \id), K_i, 1)$ 
is understood to be a subbasic set for $\Ci(M,\R^d)$.
So $\Ci(M,N)$ is open in the strong topology.
\end{proof}

\begin{theorem}\label{th:D-vs-strong}
Let $M$ and $N$ be smooth manifolds.
Then the $D$-topology on $\Ci(M,N)$ is contained in the strong topology.
\end{theorem}

\begin{proof}
Choose an embedding $N \hookrightarrow \R^d$, 
and let $U$ be an open tubular neighborhood of $N$ in $\R^d$, 
so that the inclusion $i:N \ra U$ has a smooth retract $r:U \ra N$. 
Since $i$ and $r$ induce continuous maps in both the strong topology
(\cite[Exercise 10, page~65]{Hi}) and the $D$-topology (an easy argument),
$C^\infty(M,N)$ is a subspace of $C^\infty(M,U)$ when both are 
equipped with either of these topologies.
So if these topologies agree on $\Ci(M,U)$, then they agree on 
$\Ci(M,N)$.
Thus it suffices to prove the result when $N$ is open in $\R^d$.
Assume that this is the case.

We first prove that the strong topology on $\Ci(M,\R^d)$ has a basis
of convex sets.
If $A := \cap_i \, N^r(f,(U_i,\phi_i),(V_i,\psi_i),K_i,\epsilon_i)$
is a basic open set of the strong topology, as described at the beginning
of Section~\ref{se:function-spaces}, and if $g \in A$, then by the proof
of Lemma~\ref{le:convex},
\[ 
  g \in \cap_i \, N^r(g,(U_i,\phi_i),(\R^d,\id),K_i,\epsilon_i''') \sse A ,
\]
which shows that $A$ is covered by convex strong open sets.

By Lemma~\ref{le:strong-open-subspace}, $\Ci(M,N)$ is open in $\Ci(M,\R^d)$,
so it too has a basis of convex sets.
Thus, by Lemma~\ref{le:locally-convex}, 
the $D$-topology on $\Ci(M,N)$ is contained in the strong topology.
\end{proof}

\begin{corollary}\label{M-cpt}
Let $M$ and $N$ be smooth manifolds with $M$ compact.
Then the $D$-topology on $C^\infty(M,N)$ coincides with the weak topology.
\end{corollary}

\begin{proof}
The $D$-topology is trapped between the weak topology (Proposition~\ref{weakvsD})
and the strong topology (Theorem~\ref{th:D-vs-strong}), and these coincide
when $M$ is compact.
\end{proof}

Here is one application of our results:

\begin{corollary}\label{D-open}
Let $M$ be a smooth compact manifold,
and let $\Diff(M)$ be the set of all diffeomorphisms from $M$ to itself
with the subset diffeology of $C^\infty(M,M)$.
Then $\Diff(M)$ is $D$-open in $C^\infty(M,M)$.
Hence, $\Diff(M)$ is an embedded subset of $C^\oo(M,M)$~(Definition~\ref{de:embedded}).
\end{corollary}

\begin{proof}
As mentioned in Corollary~\ref{M-cpt}, when $M$ is compact, the weak,
strong and $D$-topo\-lo\-gies on $\Ci(M,M)$ all coincide.
The first claim is then the restatement of~\cite[Theorem~2.1.7]{Hi},
and the second part follows from Lemma~\ref{l2}.
\end{proof}

Similarly, many results in~\cite[Chapter 2]{Hi} can be translated into
results for the $D$-topology.

\medskip

When $M$ is non-compact and $N$ has positive dimension, the weak topology is
different from the strong topology~\cite[pp.~35--36]{Hi}.
Since the weak topology and the $D$-topology coincide for $C^{\infty}(M, \R^n)$,
it follows that the $D$-topology and the strong topology are different
for $C^{\infty}(M, \R^n)$ when $M$ is non-compact.
We can make this explicit in the next example.

\begin{example}\label{cinfty}
It is not hard to show that the strong topology on $C^\infty(\R,\R)$ has a basis
$\{B^k_\delta(f) \mid k \in \N,\, \delta:\R \ra \R^+ \text{ continuous},\, f \in C^\infty(\R,\R) \}$,
where $B^k_\delta(f)=\{ g \in C^\infty(\R,\R) \mid\! \sum_{i=0}^k \, (f^{(i)}(x)-g^{(i)}(x))^2<\delta(x)$
for each $x \in \R \, \}$.
On the other hand, since the $D$-topo\-logy agrees with the weak topology on $C^\infty(\R,\R)$, it has a basis
$\{\tilde{B}^k_\epsilon(f) \mid k \in \N, \epsilon \in \R^+, f \in C^\infty(\R,\R)\}$,
where $\tilde{B}^k_\epsilon(f)=\{ \, g \in C^\infty(\R,\R) \mid \sum_{i=0}^k \, (f^{(i)}(x)-g^{(i)}(x))^2<\epsilon$
for each $x$ in $[-k,k]\}$.
It follows that the strong topology is strictly finer than the $D$-topology on $C^\infty(\R,\R)$.
\end{example}

On the other hand, it can be the case that the $D$-topology is different from the
weak topology but agrees with the strong topology.  For example, this happens
in case (1) of Example~\ref{weak-neq-D}, where it is easy to see that the
strong topology is also discrete.
\medskip

\begin{remark}
The book~\cite{KM} also studies function spaces between %
smooth manifolds, but uses a different smooth structure on the function space
to ensure that the resulting object has the desired local models.
By~\cite[Lemma~42.5]{KM}, their smooth structure has fewer smooth curves
than the diffeology studied here, and as a result the natural topology
discussed in~\cite[Remark~42.2]{KM} is larger than the $D$-topology.
In fact, according to that remark, it is larger than the strong topology
(which is called the $WO^\oo$-topology in~\cite{KM}).
\end{remark}

\appendix
\section{The weak topology on function spaces}\label{app}
\normalsize

In this Appendix, our goal is to prove a theorem about the weak topology on
function spaces which is analogous to the Special Curve Lemma
(Lemma~\ref{lem:scl}).  This is Theorem~\ref{th:weak-scl}.
Before proving the theorem, we collect together and prove some
basic results about the weak topology on function spaces,
and state the following lemma.

\begin{lemma}\label{le:charts}
Let $M$ and $N$ be smooth manifolds.
\begin{enumerate}[(a)]
\item There exists a locally finite countable atlas
$\{(U_i,\phi_i)\}_{i \in \N}$ of $M$ and a compact set $K_i \sse U_i$, for each $i$,
such that $M = \cup_i \, K_i^o$, where $K_i^o$ denotes the interior of $K_i$.
\item For any smooth map $f : M \to N$, there exist $\{(U_i, \phi_i, K_i)\}_{i \in \N}$
as in (a) and a countable atlas $\{(V_i,\psi_i)\}_{i \in \N}$ of $N$
such that $f(K_i) \sse V_i$ for each $i$. \qed
\end{enumerate}
\end{lemma}

Recall that for $M$ and $N$ smooth manifolds, the weak topology
on $\Ci(M,N)$ has as subbasic neighbourhoods the sets 
$N^r(f, (U,\phi), (V,\psi), K, \epsilon)$ described at
the beginning of Section~\ref{se:function-spaces}.

\begin{lemma}\label{le:weak-contains-compact-open}\label{le:wcco}
Let $M$ and $N$ be smooth manifolds.
Then the weak topology on $\Ci(M,N)$ contains the compact-open
topology.
\end{lemma}

\begin{proof}
Consider $A(K,W) = \{ g \in \Ci(M,N) \mid g(K) \sse W\}$ where
$K \sse M$ is compact and $W \sse N$ is open.
Let $f \in A(K,W)$.
Choose charts for $M$ and $N$ and compact sets $K_i$ as described in Lemma~\ref{le:charts}(b).
Choose $j$ so that $K \sse \cup_{i=1}^j \, K_i$.
Then 
\[ f \in \cap_{i=1}^j \, N^0(f, (U_i, \phi_i), (V_i \cap V, \psi_i), K_i \cap K, 1) \sse A(K,W) , \]
so $A(K,W)$ is open in the weak topology.
\end{proof}

\begin{lemma}\label{le:convex}
Let $M$ be a smooth manifold. 
The sets $N^r(f, (U,\phi), (\R^d,\id), K, \epsilon)$, where
$r \in \N$, $f \in \Ci(M,\R^d)$, $(U,\phi)$ is a chart of $M$, $K \sse U$ is compact
and $\epsilon > 0$, form a subbasis for the weak topology on $\Ci(M,\R^d)$.
In particular, the weak topology on $\Ci(M,\R^d)$ has a basis of convex sets.
\end{lemma}

\begin{proof}
Consider a subbasic set $A := N^r(f, (U,\phi), (V,\psi), K, \epsilon)$
containing a function $g$.
First observe that $g \in A' := N^r(g, (U,\phi), (V,\psi), K, \epsilon') \sse A$
for some $\epsilon'$, since these sets are determined by comparing
finitely many norms on a compact set.
One can then show that $A'' := N^r(g, (U,\phi), (V,\id), K, \epsilon'') \sse A'$
for some $\epsilon''$, using bounds on the derivatives of $\psi$ on $g(K)$.
Finally, we claim that $A''' := N^r(g, (U,\phi), (\R^d,\id), K, \epsilon''') \sse A''$
for some $\epsilon'''$.
To see this, cover $g(K)$ by finitely many open balls $B_1, \ldots, B_n$ such that
$2 B_{\ell} \sse V$ for each $\ell$, and let $\epsilon'''$ be the minimum of the
radii and $\epsilon''$.
Then if $h \in A'''$ and $x \in K$, we have $g(x) \in B_{\ell}$ for some $\ell$
and $|g(x) - h(x)| < \epsilon'''$, so $h(x) \in 2 B_{\ell} \sse V$.
\end{proof}

For $N$ open in $\R^d$, we will implicitly use that the inclusion map induces a continuous map
$C^\oo(M,N) \sse C^\oo(M,\R^d)$ in the weak topologies, which follows from the fact that
the weak topology is functorial in the second variable (see \cite[Exercise 10(a), page~64]{Hi}).
(In fact, the weak topology and the subspace topology on $C^\oo(M,N)$ agree,
but we won't need this.)
Although $C^\oo(M,N)$ need not be an open subset of $C^\oo(M,\R^d)$, 
it has the following weaker property.

\begin{lemma}\label{le:app}
Let $M$ be a smooth manifold and let $N$ be an open subset of $\R^d$. 
If $f \in C^\oo(M,N)$ and $K$ is a compact subset of $M$, then there is 
a convex basic weak $C^\oo(M,\R^d)$-neighborhood of $f$ whose elements map $K$ into $N$. 
\end{lemma}

\begin{proof}
The set $\{ g \in \Ci(M,\R^d) \mid g(K) \sse N \}$ is open in the compact-open
topology on $\Ci(M, \R^d)$ and so is open in the weak topology by Lemma~\ref{le:wcco}.
By Lemma~\ref{le:convex}, the weak topology on $\Ci(M,\R^d)$ has a basis of convex sets.
Thus any $f : M \to N$ has such a convex basic set as a weak neighbourhood.
\end{proof}

\begin{theorem}\label{th:weak-scl}
Let $M$ and $N$ be smooth manifolds.
Suppose $p: \R \to C^\oo(M, N)$ is weak continuous and 
$t_r$ is a sequence of real numbers converging to zero.
Then there is a subsequence $t_{r_j}$ and a smooth curve $q: \R \to C^\oo(M, N)$ 
such that $q(2^{-j}) = p(t_{r_j})$ for each $j$ and $q(0) = p(0)$.
\end{theorem}

\begin{proof}
We first reduce to the case where $N$ is open in $\R^d$.
As in Theorem~\ref{th:D-vs-strong}, choose an embedding $N \hookrightarrow \R^d$, 
and let $U$ be an open tubular neighborhood of $N$ in $\R^d$, 
so that the inclusion $i:N \ra U$ has a smooth retract $r:U \ra N$. 
By~\cite[Exercise 10(a), page~64]{Hi}, the map $\R \to \Ci(M, U)$ sending $t$ to $i \circ p(t)$ is weak continuous,
so if the Theorem holds for $\Ci(M, U)$, then there is a smooth curve $q: \R \to C^\oo(M, U)$ 
such that $q(2^{-j}) = i \circ p(t_{r_j})$ for each $j$ and $q(0) = i \circ p(0)$.
Then the map sending $t$ to $r \circ q(t)$ is smooth, 
$r \circ q(2^{-j}) = p(t_{r_j})$ for each $j$, and $r \circ q(0) = p(0)$,
so we are done.
Thus we may assume that $N$ is open in $\R^d$.

If $t_r$ is eventually constant, we may take $q$ to be a constant function, 
so suppose it is not.
Choose charts $(U_k,\phi_k)_{k=1}^\oo$ for $M$ and compact sets $K_k \sse U_k$ as described in Lemma~\ref{le:charts}(a). 
Let $f = p(0)$. For $j = 1,2,\dots$, the sets,
\[
  A_j = \cap_{k=1}^j N^j(f,(U_k,\phi_k),(\R^d,\id),K_k,2^{-(j+1)^2})
\]
are weak $C^\oo(M,\R^d)$-neighborhoods of $f$ so we may choose a 
strictly monotone subsequence $t_{r_j}$ such that $p(t_{r_j}) \in A_j$ for each $j$. 
Set $f_j = p(t_{r_j})$.
Now compose $p$ with a continuous function taking $2^{-j}$ to $t_{r_j}$ for each $j$ 
to obtain a weak continuous function $p_0$ that satisfies 
$p_0(2^{-j}) = f_j$ for $j = 1,2,\dots$ and $p_0(0) = f$. 

Fix $k$.
By Lemma~\ref{le:app}, for each $t \in [0,1]$, there is a convex neighborhood of $p_0(t)$ 
whose elements map $K_k$ into $N$.
By compactness, there is a $\delta_k > 0$ such that any subinterval of $[0,1]$ 
of length at most $2\delta_k$ is mapped by $p_0$ into one of these neighborhoods.
Thus, for each $t$, any convex combination of elements in 
$p([t-\delta_k,t+\delta_k] \cap [0,1])$ maps $K_k$ into $N$.
Let $\tau_0, \tau_1, \dots$ be the strictly decreasing sequence obtained by ordering the set 
$\{1, 1/2, 1/4,\dots\} \cup \{\delta_k, 2\delta_k, \dots, \lfloor1/\delta_k\rfloor\delta_k\}$.
Note that $\tau_0 = 1$ and $\tau_{j-1}-\tau_j \le \delta_k$ for $j = 1,2,\dots$.

Fix a non-decreasing $\mu \in C^\oo(\R,[0,1])$ such that $\mu=0$ in a neighborhood of 
$(-\oo,0]$ and $\mu=1$ in a neighborhood of $[1,\oo)$.
Let $\cM_\ell = 1 + 2 \max_{\ell'\le \ell} \max_{t\in[0,1]} |\mu^{(\ell')}(t)|$.

Define $q_k: \R \to C^\oo(M, \R^d)$ by $q_k(t) = p_0(0)$ for $t \le 0$, $q_k(t) = p_0(1)$ for $t \ge 1$, and 
\[
  q_k(t) = p_0(\tau_j) + \mu\left(\tfrac{t-\tau_j}{\tau_{j-1}-\tau_j}\right)(p_0(\tau_{j-1})-p_0(\tau_j))
\] 
for $\tau_j \le t \le \tau_{j-1}$, $j = 1,2,\dots$.
Note that for each $t \in (0,1]$, $q_k(t)$ is a convex combination of elements of 
$p_0([t-\delta_k,t+\delta_k] \cap [0,1])$.
Clearly, $q_k$ is constant on $(-\oo,0]$ and constant on $[1,\oo)$.
The choice of $\mu$ ensures that it is constant in a neighborhood of $\tau_{j-1}$ 
for each $j$ and smooth on $(\tau_j, \tau_{j-1})$.
Thus, $q_k$ is smooth on $(\R \setminus \{0\}) \times M$.
To see that it is also smooth at $t=0$, fix a positive integer $\kappa$, 
set $F = f \circ \phi_\kappa^{-1}$, $F_j = f_j \circ \phi_\kappa^{-1}$ for each $j$, 
and $Q(t,s) = q_k(t)(\phi_\kappa^{-1}(s)) - F(s)$.
It will suffice to show that all partial derivatives of $Q$ exist and equal zero on 
$S := \{0\} \times \phi_\kappa(K^o_\kappa)$.
Certainly $Q=0$ there and if $\cD$ is any composition of partial differentiation operators 
such that $\cD Q$ vanishes on $S$ then the partial derivative of $\cD Q$ 
with respect to any of $s_1,\dots,s_m$ also vanishes there.
To complete the induction, it is enough to show that the partial derivative of 
$\cD Q$ with respect to $t$ also vanishes on $S$. 

Where $Q$ is $C^\oo$, the order of mixed partials is unimportant, so 
$\cD Q = D_t^\ell D_s^i Q$ off $S$ for some $\ell \ge 0$ and some multi-index $i$.
Choose $J$ so that $2^{-J} < \delta_k$.
Then $2^{-J},2^{-J-1},2^{-J-2},\dots$ is a tail of the sequence $\tau_0,\tau_1,\dots$.
So if $j>J$ and $2^{-j} \le t\le 2^{1-j}$, then
\[
  q_k(t) = f_j + \mu(2^jt-1)\,(f_{j-1}-f_j)
\] 
and, for $s \in \phi_\kappa(U_\kappa)$, 
\[
(D_t^\ell D_s^i Q)(t,s) = 
  \begin{cases} 
    (D_s^i(F_j-F))(s) + \mu(2^jt-1)\,(D_s^i(F_{j-1}-F_j))(s),  & \ell = 0,\\
    \mu^{(\ell)}(2^jt-1) \, 2^{\ell j} \, (D_s^i(F_{j-1}-F_j))(s), & \ell \ge 1.
  \end{cases}
\]
If $j > \max(J,\kappa,|i|,\ell+2)$, then 
\[
  f_j \in A_j \sse N^j(f,(U_\kappa,\phi_\kappa),(\R^d,\id),K_\kappa,2^{-(j+1)^2}) \]
and
\[
  f_{j-1} \in A_{j-1} \sse N^{j-1}(f,(U_\kappa,\phi_\kappa),(\R^d,\id),K_\kappa,2^{-j^2}),
\]
so 
\[
  |D_s^i(F_j-F)| \le 2^{-(j+1)^2} \le 2^{-j^2}
  \quad\text{and}\quad
  |D_s^i(F_{j-1}-F)| \le 2^{-j^2}
\]
on $\phi_{\kappa}(K_\kappa)$.
Thus, for any $s \in \phi_k(K_\kappa^o)$,
\[
|(\cD Q)(t,s)-(\cD Q)(0,s)| = |(D_t^\ell D_s^iQ)(t,s)| \le \cM_\ell \, 2^{\ell j} \, 2^{-j^2} \le \cM_\ell \, t^2,
\]
where we have used that $\ell < j - 2$ in the last inequality.
Since $j$ can be arbitrarily large, this inequality holds for all sufficiently small $t$,
so the partial derivative of $\cD Q$ with respect to $t$ (from the right) exists and equals zero.
The partial derivative from the left is trivially zero.
This completes the induction and the proof that $q_k$ is smooth.

Before allowing $k$ to vary, observe that $q_k(\tau_j) = p_0(\tau_j)$ for each $j$ and, 
in particular, $q_k(2^{-j}) = p_0(2^{-j}) = f_j$ for each $j$.

Let $(\nu_k)_{k=1}^\oo$ be a smooth partition of unity on $M$ with 
$\nu_k$ supported in $K_k^o$ and define $q$ by $q(t)(x) = \sum_{k=1}^\oo \nu_k(x) q_k(t)(x)$.
Then $q: \R \to C^\oo(M,\R^d)$ is a smooth curve such that $q(2^{-j}) = f_j = p(t_{r_j})$ 
for each $j$, and of course $q(0)=p(0)$.
It remains to show that $q(t)$ takes values in $N$ for each $t \in \R$.
Let $x \in M$.
There are finitely many $k$ such that $\nu_k(x) \ne 0$; 
among them, choose $k'$ so that $\delta_{k'}$ is as large as possible.
Then, for any $t$ and any $k$ such that $\nu_k(x) \ne 0$, $q_k(t)$ is a convex combination 
of elements of $p_0([t - \delta_{k'}, t + \delta_{k'}] \cap [0,1])$.
Thus, $\sum_{k=1}^\oo \nu_k(x) q_k(t)$ is also a convex combination of elements 
of $p_0([t-\delta_{k'}, t+\delta_{k'}] \cap [0,1])$ and therefore maps $K_{k'}$ to $N$.
But $\nu_{k'}(x) \ne 0$, so $x \in K_{k'}$.
Hence,  $\sum_{k=1}^\oo \nu_k(x) q_k(t)(x) \in N$, that is, $q(t)(x) \in N$.
We conclude that $q: \R \to C^\oo(M,N)$.
This completes the proof.
\end{proof}

\end{document}